\documentclass[12pt]{amsart}
\usepackage{amsmath}
\usepackage{graphicx}
\usepackage[colorlinks=true]{hyperref}
\usepackage{color}
\usepackage{accents}
\newcommand{\beq}{\begin{eqnarray*}}
\newcommand{\feq}{\end{eqnarray*}}
\newcommand{\beqn}{\begin{eqnarray}}
\newcommand{\feqn}{\end{eqnarray}}

\newtheorem{theorem}{Theorem}[section]
\newtheorem{lemma}[theorem]{Lemma}

\theoremstyle{definition}

\theoremstyle{remark}

\numberwithin{equation}{section}

\newtheorem*{theorem*}{Theorem}

\begin{document}
\title[Wave breaking in a class of non-local conservation laws]{Wave breaking in a class of non-local conservation laws}

\author{Yongki Lee$^\dag$}
\address{$^\dag$Department of Mathematical Sciences, Georgia Southern University, Statesboro,  Georgia 30458}
\email{yongkilee@georgiasouthern.edu}
\begin{abstract} 
For models describing water waves, Constantin and Escher \cite{CE98}'s works have long been considered as the cornerstone method for proving wave breaking phenomena. Their rigorous analytic proof shows that if the lowest slope of flows can be controlled by its highest slope initially, then the wave-breaking occur for the Whitham-type equation.  Since this breakthrough, there have been numerous refined wave-breaking results established by generalizing  the kernel which describes the dispersion relation of water waves. Even though the proofs of these involve a system of coupled Riccati-type differential inequalities, however, little or no attention has been made to a generalization of this Riccati-type system. In this work, from a rich class of non-local conservation laws, a Riccati-type system that governs the flow's gradient is extracted and investigated. The system's leading coefficient functions are allowed to change their values and signs over time as opposed to the ones in many of other previous works are fixed constants. Up to the author's knowledge, the blow-up analysis upon this structural generalization is new and is of theoretical interest in itself as well as its application to various non-local flow models. The theory is illustrated via the Whitham-type equation with nonlinear drift. Our method is applicable to a large class of non-local conservation laws.

\end{abstract}
\maketitle

\section{Introduction}
Wave breaking phenomena$-$bounded solutions with unbounded derivatives$-$usually appeared in water wave and traffic flow models. The aim of the present work is to investigate the wave breaking phenomena for a class of non-local conservation laws,
\begin{equation}\label{1main}
\left\{
  \begin{array}{ll}
    \partial_t u + \partial_x F(u, \bar{u}) =0, & t>0, x \in \mathbb{R}, \\
    u(0,x)=u_0 (x), &   x\in \mathbb{R},\hbox{}
  \end{array}
\right.
\end{equation}
where $u$ is the unknown, $F$ is a given smooth function, and $\bar{u}$ is given by
\begin{equation}\label{ubar}
\bar{u}(t,x)=(K*u)(t,x)=\int_{\mathbb{R}} K(x-y)u(t,y) \, dy,
\end{equation}
where $K$ to be chosen later. The nonlinear advection couples both local and nonlocal mechanism. 
This class of conservation laws, identified in \cite{LL15},  appears in several applications including  traffic flows \cite{SK06, KP09},  the collective motion of biological cells \cite{DS05, MDS08}, dispersive water waves \cite{GW74, Holm, Liu0},   the radiating gas motion \cite{Ha71, LT02}, high-frequency waves in relaxing medium \cite{Hunter, Vak1, Vak2} {\color{black} and the kinematic sedimentation model \cite{Kynch, KZ99, BBKT}.}

One of aspects of our interest in this class of non-local conservation laws is the critical threshold phenomena, as in many other hyperbolic balance laws. As is known that the typical well-posedness result asserts that either a solution of a time-dependent PDE exists for all time or else there is a finite time such that some norm of the solution becomes unbounded as the life span is approached. The natural question is whether there is a critical threshold for the initial data such that the persistence of the $C^1$ solution regularity depends only on crossing such a critical threshold. This concept of critical threshold(CT) and associated methodology is originated and developed in a series of papers by Engelberg, Liu and Tadmor \cite{ELT01, LT02} for a class of Euler-Poisson equations.

The wave breaking criterion is established for several model equations of the form \eqref{1main} and the Euler-Poisson system. There are some distinguished special cases of \eqref{1main} with the kernel $K$:\\
$\bullet$ A traffic flow model with Arrhenius look-ahead dynamics \cite{SK06}
$$u_t + [u(1-u)e^{-\bar{u}}]_x =0,$$
corresponding to \eqref{1main} with $F(u, \bar{u})=u(1-u)e^{-\bar{u}}$ and 
\begin{equation}\label{traffic_vanilla_kernel}
K(r)=
\left\{
  \begin{array}{ll}
K_0 / \gamma, & $if$ -\gamma \leq r \leq 0, \\
0, &   $otherwise$;\hbox{}
  \end{array}
\right.
\end{equation}
$\bullet$ A shallow water model proposed by Whitham \cite{GW74}
$$u_t + \frac{3c_0}{2h_0}uu_x + \bar{u}_x =0,$$
corresponding to \eqref{1main} with $F(u, \bar{u})=\frac{3c_0}{4h_0} u^2 + \bar{u}$ and $K(r)=\frac{\pi}{4}\exp(-\pi |r|/2)$;\\
$\bullet$ The hyperbolic Keller-Segel model with logistic sensitivity \cite{DS05}
\begin{equation}
\left\{
  \begin{array}{ll}
     u_t + [u(1-u)\partial_x S]_x=0,\\
    - S_{xx} + S = u, & \hbox{}
  \end{array}
\right.
\end{equation}
corresponding to \eqref{1main} with $F(u, \bar{u})=u(1-u)\bar{u}$ and $K(r)=\partial_r (e^{-|r|}/2)$;\\
$\bullet$ A nonlocal dispersive equation modeling particle suspensions \cite{JR90, RK89, KZ99}
$$u_t + u_x + ((K_a * u)u)_x =0,$$
corresponding to \eqref{1main} with $F(u, \bar{u})=u + \bar{u}u$, $K_a (r)=a^{-1} K(r/a)$ and
\begin{equation}
K(r)=
\left\{
  \begin{array}{ll}
  2/(3(r^2 /4 -1)) & $if$ \ |r| < 2, \\
0, &   $otherwise$.\hbox{}
  \end{array}
\right.
\end{equation}
$\bullet$ A traffic flow model with nonlocal-concave-\emph{convex} flux \cite{L-P, HA91, TL03}
$$u_t + [u(1-u)^2 e^{-\bar{u}}]_x =0,$$
corresponding to \eqref{1main} with $F(u, \bar{u})=u(1-u)^2 e^{-\bar{u}}$, and the kernel in \eqref{traffic_vanilla_kernel}.

Blow-up techniques are quite particular to each type of equation, there is no general method. Motivated by Seliger \cite{S68}, Constantin and Escher \cite{CE98}'s ingenious works, our goal in this work is to establish a quite representative sample of methods to accomplish wave-breaking for a class of nonlocal conservation laws.

Seliger \cite{S68} and Constantin and Escher \cite{CE98}'s works have been considered as the cornerstone method for proving wave breaking for Whitham-type equation,
\begin{equation}\label{1WT}
u_t + uu_x + \int_{\mathbb{R}} K(x-z)u_{z}(t,z) \, dz =0,
\end{equation}
with bounded and integrable kernel. Seliger's ingenious argument is formally tracing the dynamics of
\begin{equation}\label{1m12}
m_1 (t): = \min_{x\in \mathbb{R}} [u_x (t,x)], \ \ \ and \ \ \ m_2 (t): = \max_{x\in \mathbb{R}} [u_x (t,x)],
\end{equation}
along two different curves  $\xi_1 (t):=\mathrm{argmin}_{x} [u_x (t,x)]$ and  $\xi_2 (t):=\mathrm{argmax}_x [u_x (t,x)]$.
More precisely, by differentiating \eqref{1WT} and evaluating it at $x=\xi_i (t)$,  he obtained the following Riccati type equation
\begin{equation}\label{1m}
\frac{D}{Dt}m_i = -m^2 _i + \cdots, \ \  i=1,2
\end{equation}
which yield the desired wave breaking. Here $\frac{D}{Dt}$ is the convective derivative, i.e. $\frac{D}{Dt}:=\partial_t + u\partial_x$. To carry out Seliger's formal analysis, one needs to assume that the curves $\xi_i (t)$, $i=1,2$ are smooth. This additional strong assumption was shown unnecessary later by Constantin and Escher.

Following their methods, the author \cite{LL15} identified sub-thresholds for finite time shock formation in a class of nonlocal conservation laws, which we summarize here.
Under the following two assumptions: \\
($H_1$).     $F \in C^3 (\mathbb{R}, \mathbb{R})$, and the kernel $K(r) \in W^{1,1} $ satisfying
$$
K(r)=
\left\{
  \begin{array}{ll}
    Nondecreasing, & \hbox{$r\leq 0$,} \\
    0, & \hbox{$r>0$.}
  \end{array}
\right.
$$
($H_2$).   $F(0, \cdot)=F(m, \cdot)=0$ and
$$F_{uu} < 0 , \ F_{\bar{u}\bar{u}} > 0, \ \ F_{\bar{u}}<0\quad \text{for}\quad u\in [0, m],
$$
we have the following result.
\begin{theorem} $($\cite{LL15}$)$ Consider \eqref{1main} with \eqref{ubar} under assumptions  ($H_1$)-($H_2$).  If $u_0 \in H^2$ and $0  \leq u_0 (x) \leq m $ for all $x \in \mathbb{R}$, then there exists a non-increasing function
$\lambda(\cdot)$ such that if
$$\sup_{x \in \mathbb{R}} [u' _{0}(x)] > \lambda(\inf_{x \in \mathbb{R}} [u' _{0}(x)] ),$$
then $u_x$ must blow up at some finite time.
\end{theorem}

In our proof \cite{LL15}, the condition $F_{uu} <0$ in ($H_2$) is crucial. That is, the global flux function $F(u, \bar{u})$ must be either concave up only or concave down only during a solution's  life span. Essentially the same stipulation is inherent for the models in \cite{CE98, MLQ16} and many blow-up analysis that use a Riccati-type equation e.g., \cite{LL01, L17, TT14}. This is because $F_{uu}$  serve as the leading coefficient function of the Riccati type equation and  sign changes of $F_{uu}$ flip over the phase diagram of $d:=u_x$ (setting aside other factors) which governed by
$$\frac{D}{Dt} d(t) =  -F_{uu}(t) d^2 (t) +\cdots.$$

If the global flux function changes its concavity over time, like some models in next sections, then one is facing the following system after evaluating the above equation at $\xi_i (t)$:
\begin{equation}\label{1gm}
\frac{D}{Dt}m_i (t) = g_i (t) m^2 _i (t)  + \cdots, \ \  i=1,2,
\end{equation}
where $m_i$ defined in \eqref{1m12} and $g_i (t)$ is the evaluation of $-F_{uu}$ at $\xi_i (t)$. Here $g_{i} (t)$ are allowed to change their signs as $t$ evolves(due to the flux's concavity changing).

The system in \eqref{1gm} structurally generalize the system \eqref{1m}, which has $-1$ as its dominating coefficient. There has been many efforts to generalize/calibrate a kernel  $K$ which describes the dispersion relation of water waves, and changing kernel is corresponds to the changes of non-dominating coefficient functions in  \eqref{1m} or \eqref{1gm}.  In spite of the fact that \eqref{1gm} displays the dynamics of a flow gradient for rich class of non-local conservation laws \eqref{1main},  little or no attention has been made to this generalization of dominating coefficients in \eqref{1gm}. 

In the present work, we investigate the blow-up condition of the system \eqref{1gm}, which in turn identifies the wave breaking phenomena of \eqref{1main} with concavity changing $F$. As an application of our method, 
we discuss the wave breaking of the following Whitham-type equation  (13.131) in \cite{GW74}:
$$u_t + \big{(} 3\sqrt{g(h_0 +u)} -2 \sqrt{gh_0}  \big{)}u_x + \int_{\mathbb{R}} K(x-y)u_y (t, y) \, dy =0.$$
In the next section, the significances of the above models are discussed along with the recent studies on similar models. We want to point out that our argument work equally well with not only the above equations, but also a large class of non-local conservation law in \eqref{1main}, as well as many generalized kernels.

The remainder of the paper is organized as follows. In Section 2, we introduce the goal of this paper. In Section 3, we bring the ``original"  Whitham-type equation from \cite{GW74}, which has nonlinear and more physically relevant wave drifting term. The differences between the Whitham-type equation and this ``original" Whitham-type equation are discussed, and challenges in blow-up analysis on this model is introduced. Then the wave-breaking condition of the model is identified.  Finally, Sections  4 is devoted to provided the detailed proof of theorem in Sections 3.

We want to emphasize the method and not the technicalities. For this reason we focus on the cases of  simple kernels.  it is not hard to see then how the method applies to the more complicated kernel cases handled in e.g., \cite{MLQ16, L-P, LL15}.

\section{Hightlights of this paper}
In this section, we briefly discuss our methods of the blow-up analysis for a class of non-local conservation laws, and the goal of this paper.
\begin{equation}\label{mmain}
\left\{
  \begin{array}{ll}
    \partial_t u + \partial_x F(u, \bar{u}) =0, & t>0, x \in \mathbb{R}, \\
    u(0,x)=u_0 (x), &   x\in \mathbb{R},\hbox{}
  \end{array}
\right.
\end{equation}
along with application to the various models.

Let $d=u_x$, then applying $\partial_x$ to the first equation of \eqref{mmain} leads 
\begin{equation}\label{key_R}
\dot{d}:=(\partial_t + F_1 \partial_x)d = - [F_{11} d^2 +2F_{12} \bar{u}_x d + F_{22}\bar{u}^2 _x + F_{2} \bar{u}_{xx}].
\end{equation}
We trace the dynamics of $F_{11}$ along the characteristic. Since
\begin{equation*}
\begin{split}
\partial_t F_{11} &= F_{111} u_t + F_{112} \bar{u}_t\\
&=F_{111} \cdot (-F_1 u_x -F_2 \bar{u}_x) + F_{112} \bar{u}_t,
\end{split}
\end{equation*}
and
$$F_{1}\partial_x F_{11} = F_{1} (F_{111} u_x + F_{112}\bar{u}_x),$$
it follows that
\begin{equation}
\dot{F}_{11} :=(\partial_t + F_1 \partial_x) F_{11}=(-F_{111} F_2 + F_1 F_{112}  )\bar{u}_x + F_{112} \bar{u}_t.
\end{equation}
Here $F_{i}$ denotes the derivative of $F(u, \bar{u})$ with respect to its $i$-th component. 

We should point out that $u_x$ terms are cancelled out in $\dot{F}_{11}$. That is, the rate of change of $F_{11}$ along the characteristic does not carry the rate of change of $u$. This enable us to perform the blow-up analysis on the Riccati-type ODE in \eqref{key_R} even though the leading coefficient function $F_{11}$ can change its sign. 

In many cases $\dot{F}_{11}$ has either uniform constant bounds or bounds that depend on $u$ only.  We first list some examples that $\dot{F}_{11}$ is bounded by some constant:\\
$\bullet$ The hyperbolic Keller-Segel model with logistic sensitivity \cite{DS05}
corresponds to \eqref{mmain} with $F(u, \bar{u})=u(1-u)\bar{u}$ and $K(r)=\partial_r (e^{-|r|}/2)$ has
$$\dot{F}_{11}\geq -\frac{3}{4},$$
where $F_{11}=-2\bar{u}$ can change its sign. The detailed derivation of the bound and blow-up result can be found in \cite{LL15-1}.\\
$\bullet$ A traffic flow model with Arrhenius look-ahead dynamics \cite{SK06} corresponds to \eqref{mmain} with $F(u, \bar{u})=u(1-u)e^{-\bar{u}}$ and 
\begin{equation}\label{traffic_vanilla_kernel_2}
K(r)=
\left\{
  \begin{array}{ll}
K_0 / \gamma, & $if$ -\gamma \leq r \leq 0, \\
0, &   $otherwise$.\hbox{}
  \end{array}
\right.
\end{equation}
In this model, $F_{11} = -2 e^{-\bar{u}}$ does not change its sign. The blow-up results of this model can be found in \cite{LL11, LL15}. One may obtain a sharper blow-up condition because
$$\dot{F}_{11} = e^{-\bar{u}} \{ 2(1-2u)\bar{u}_x e^{-\bar{u}} +2 \bar{u}_t \}$$
and all terms herein have some uniform bounds inherited from \emph{a priori} bounds $0\leq u \leq 1.$ 
$\bullet$ A traffic flow model with a nonlocal-concave-\emph{convex} flux \cite{L-P} corresponds to \eqref{mmain} with 
$$F(u, \bar{u})=u(1-u)^2 e^{-\bar{u}}$$ and the kernel in \eqref{traffic_vanilla_kernel_2}.
In this model, $F_{11} = (-4 + 6u)e^{-\bar{u}}$ can change its sign. However, since there exists a constant $\alpha$ such that
$$|\dot{F}_{11} | \leq \alpha,$$
one can obtain the blow-up condition of $u_x$ \cite{L-P}.

For the above models, our blow-up analysis argument can be summarized as follows. Consider the Riccati type ODE in \eqref{key_R}, i.e.,
\begin{equation}\label{key_R_1}
\dot{d} =  -F_{11} (d - R_1)(d - R_2),
\end{equation}
where $R_{1} \leq R_{2}$ are the roots of the quadratic equation in \eqref{key_R}. Since there exists $\alpha >0$ such that $|\dot{F}_{11}| \leq \alpha$, one can find $t^* >0$ such that $F_{11}$ does not change its sign for $t\in [0, t^*]$. Using \eqref{key_R_1}, we then find the set of initial data that lead to the blow-up of $d$ \emph{before} $t^*$. Since $R_1$ and $R_2$ in \eqref{key_R_1} involve some \emph{global} terms, in order to estimate these, one need to use Seliger's argument, i.e., tracing the dynamics of $\min{u_x}$ and $\max{u_x}$.

One may think an extension of this method by considering the case such that
$$|\dot{F}_{11} | \leq h(u)$$ 
for some function $h(\cdot)$ and \emph{no} uniform bound of $u$ is known.
More precisely, consider a class of nonlocal dispersive equations of the form e.g., \cite{Liu0}
$$u_t + auu_x + [Q*B(u, u_x)]_x =0.$$
This equation includes many shallow water models e.g., Camassa-Holm equation, Whitham equation, and Korteweg-de Vries equation. Considering the above nonlocal dispersive equation with some \emph{nonlinear drift}
$$u_t + f(u)u_x + [Q*B(u, u_x)]_x =0,$$
it follows that
$$\dot{d} = -f'(u)d^2 + \cdots.$$
We shall consider the cases such that $f'(u)$ changes its sign and there is no uniform bound of $u$. 
The goal of this paper is to establish a quite representative sample of methods that accomplish wave-breaking for a large class of non-local conservation laws \eqref{1main} equipped with \emph{concavity changing flux}.

\section{Breaking of waves for Whitham-type equations with non-linear drift}

The surface wave model that motivated this study is the Whitham type equation:
\begin{equation}\label{WTlinear}
\eta_t + \frac{3c_0}{2h_0}\eta \eta_x + [ K*\eta ]_x =0.
\end{equation}
The function $\eta(t,x)$ models the deflection of the fluid surface from the rest position. Here, $h_0$ denotes the undisturbed fluid depth, and $c_0 := \sqrt{g h_0}$ denotes the limiting long-wave speed where $g$ is the gravitational acceleration constant. The equation was proposed by Whitham  as an alternative to the KdV equation for the description of wave motion at the surface of a perfect fluid by simplified evolution equations. Whitham emphasized that the breaking phenomena(bounded solutions with unbounded derivatives) is one of the most intriguing long-standing problems of water wave theory, and since the KdV equation can't describe breaking, he suggested \eqref{WTlinear} with the singular kernel.

We want to point out that the linear drift term $\frac{3c_0}{2h_0}\eta$ is actually a result of \emph{linear approximation} for the nonlinear wave propagating speed. Indeed, in the nonlinear shallow water equations
$$\eta_t + h_0 u_x +(\eta u)_x =0, \ \ u_t + uu_x + g \eta_x =0,$$
a waves propagating to the right into undisturbed water of depth $h_0$ satisfies the Riemann invariant
$$u= 3\sqrt{g(h_0 +\eta)} -2 \sqrt{gh_0}.$$
With this consideration, Whitham derived the equation (13.131) in \cite{GW74}:
\begin{equation}\label{WTnonlinear}
\left\{
  \begin{array}{ll}
     \eta_t + \big{(} 3\sqrt{g(h_0 +\eta)} -2 \sqrt{gh_0}  \big{)}\eta_x + \int_{\mathbb{R}} K(x-y)\eta_y (t, y) \, dy =0, & t>0, x \in \mathbb{R}, \\
    \eta(0,x)=\eta_0 (x), &   x\in \mathbb{R}.\hbox{}
  \end{array}
\right.
\end{equation}

The main contribution of this section is to show the wave breaking of  the ``original" Whitham type equation \eqref{WTnonlinear} provided that $K$ be bounded and integrable, among other hypotheses.

In our study we set $g=h_0=1$, since these constants are not essential in our blow-up analysis. Also, for our notational convenience we use $u(t,x)$ as our unknown function. That is, we will consider
\begin{equation}\label{WTnonlinearf}
\left\{
  \begin{array}{ll}
     u_t + f(u) u_x + \int_{\mathbb{R}} K(x-y)u_y (t, y) \, dy =0, & t>0, x \in \mathbb{R}, \\
    u(0,x)=u_0 (x), &   x\in \mathbb{R}, \hbox{}
  \end{array}
\right.
\end{equation}
where $f(u)=3\sqrt{1+u} -2$.

We note that even though the reverted Whitham equation \eqref{WTnonlinearf} looks similar to the model in \eqref{WTlinear}, they are quite different in the perspectives of modeling and blow-up analysis. Let $d=\eta_x$ and apply $\partial_x$ to \eqref{WTlinear}, it follows that
\begin{equation}\label{Roriginal}
\frac{D}{Dt} d  = -\frac{3c_0}{2h_0} d^2 + \cdots,
\end{equation}
where $\frac{D}{Dt}:= (\partial_t  + \frac{3c_0 }{2h_0} \eta \partial_x)$ is the convective derivative operator. On the other hand, letting $d=u_x$ and applying $\partial_x$ to \eqref{WTnonlinearf} gives
\begin{equation}\label{Rnonlinear}
\frac{D}{Dt} d = -f'(u) d^2 + \cdots,
\end{equation}
where the convective derivative operator is defined similarly and 
$$-f'(u)=-\frac{3}{2} \frac{1}{\sqrt{1+u}}.$$

We first address the difference in modeling aspect.  In \eqref{Roriginal}, the wave breaking can occur at \emph{any height} of the wave as long as $d=u_x$ initially satisfies the blow up criteria in \cite{CE98}. In contrast to this, the leading coefficient function in \eqref{Rnonlinear} is taking into account $u$; the fluid deflection. Therefore, during the course of the proof we anticipate some \emph{balance} between wave height and its gradient that leads to the wave breaking. It is not clear which one is more physically relevant, but in \eqref{Rnonlinear} the fluid wave's height plays an important role for the wave breaking whereas the height in \eqref{Roriginal} plays no role for the wave breaking. 

In the blow up threshold analysis aspect, the model in \eqref{Rnonlinear} can be viewed as a generalization of \eqref{Roriginal}. There has been numerous effort focused on showing the wave breaking with generalized and refined Kernel $K$. Some recent works in this direction can be found in  \cite{MLQ16, H17} and reference therein. However, little or no attention has been made to the structural generalization given in \eqref{Rnonlinear}. 
Heuristically, the magnitudes of the leading coefficient functions in the Riccati equations \eqref{Roriginal} and \eqref{Rnonlinear} determine how wide the quadratic functions' graphs are. Moreover, the roots of the quadratic function will be phase switching locations of the Riccati differential equation phase diagram. For the traffic flow model with look-ahead dynamics \cite{LL15, LL11, L-P}, even though the Riccati equation has a relatively simple leading coefficient,  the roots are changing their locations in time, and generic non-locality makes estimation of these roots difficult. In \eqref{Rnonlinear}, the roots of the quadratic equation are changing over time due to $-f'(u)$ as well as the non-locality. 

If we assume that there exists a conserved quantity, for example $\|u(t) \|_{H^1}=\|u_0 \|_{H^1}$ for \eqref{WTnonlinearf} like the Camassa-Holm equation, then since
\begin{equation*}
\begin{split}
u^2 (t,x) &= \int^x _{\infty} uu_x \, dy - \int^{\infty} _x uu_x \,dy \\
&< \frac{1}{2} \int^x _{-\infty} u^2 + u^2 _x \, dy +  \frac{1}{2} \int^{\infty} _{x} u^2 + u^2 _x \, dy= \frac{1}{2}\|u(t) \|^2 _{H^1},
\end{split}
\end{equation*}
one may rewrite \eqref{Rnonlinear} as
$$\frac{D}{Dt} d < -\frac{3}{2}\frac{1}{\sqrt{1 + \|u_0 \|_{H^1} /\sqrt{2} }}d^2 +\cdots$$
and can simply obtain wave breaking result using the argument in \cite{CE98}.  However, our proof is based on estimating the rate of change of $-f'(u)$ along the same characteristic, not the non-sharp and ``may not exist" inequality
$$|u(t) | \leq \frac{ \|u_0 \|_{H^1} }{\sqrt{2}},$$
 that may lead to some non-sharp blow-up condition. Moreover, our argument work equally well with any smooth function $f(u)$ and many generalized or singular kernels $K$ in \cite{MLQ16, H17}.

To state our main result, we introduce several quantities with which we characterize the initial behavior of  profile $u(t, x)$. These are
$$m_1 (0) := \inf_{x\in \mathbb{R}} [u_x (0,x)], \ \ \ \ m_2 (0) := \sup_{x\in \mathbb{R}} [u_x (0, x)],$$
and
$$\xi_1 (0) :=\mathrm{argmin}[u' _0 (x)], \ \ \ \ \xi_2 (0):=\mathrm{argmax}[u' _0 (x)],$$
i.e., points for which $u' _0 (x)$ attains its maximum and minimum, respectively, and
$$u^M _0 := \max_{i=1,2}[u(0, \xi_i (0))], \ \ \ \ u^m _0 := \min_{i=1,2}[u(0, \xi_i (0))].$$

\begin{theorem}\label{thm_wt}
Let $u \in C^\infty ([0, T); H^{\infty}(\mathbb{R}))$ be a solution of \eqref{WTnonlinearf} before breaking. Assume $K(x)$ in \eqref{WTnonlinearf} is regular(smooth and integrable over $\mathbb{R}$) and symmetric and monotonically decreasing on $R^{+}$. If there exists a constant $\mu$ such that $-\frac{3}{2} \frac{1}{\sqrt{1+u^M _0}} \leq \mu <0$ and
\begin{equation}\label{YKcon}
\inf_{x\in \mathbb{R}} [ u' _0(x)] \leq \min \bigg{[} \frac{2K(0)}{\mu} -  \sup_{x\in \mathbb{R}} [ u' _0(x)] ,  \frac{K(0)}{\mu} + E\big{(}\mu, \kappa, u^m _0, u^M _0 \big{)} \bigg{]},
\end{equation}
then $u_x$ must blow-up at some finite time. 
Here, 
$$E\big{(}\mu, \kappa, u^m , u^M \big{)}:=\frac{1}{\mu} \bigg{(}\frac{3}{4} \frac{\| K' \|_{L^{\infty}} \| u_0 \|_{L^1}}{(\kappa(1+ u^m ))^{3/2}} \bigg{)} \bigg{(}\mu + \frac{3}{2}\frac{1}{\sqrt{1+u^M }}\bigg{)}^{-1} ,$$
and
$$\kappa(\mu, u^m, u^M):=\bigg{(} \frac{4}{3}\mu \sqrt{1+ u^m } + 2\frac{\sqrt{1+u^m }}{\sqrt{1 + u^M}} +1 \bigg{)}^{-1}.$$
\end{theorem}

Some remarks are in order regarding this result.\\
i) The wave breaking condition of Whitham's equation in \cite{CE98} is
\begin{equation}\label{WTcon}
\inf_{x\in \mathbb{R}} [ u' _0(x)] +\sup_{x\in \mathbb{R}} [ u' _0(x)] <-2K(0). 
\end{equation}
 Besides these lowest/highest initial slopes, the wave breaking for the \emph{reverted} Whitham's equation depends on several initial quantities including the initial values of  $$-f'(u(t,x))=-\frac{3}{2} \frac{1}{\sqrt{1+u(t,x)}}$$ evaluated at $\xi_i (0)$. This is natural because the key in our proof is tracing the dynamics of $-f'(u(t,x))$(along the characteristic), which serve as the leading coefficient of the Riccati equation.\\
ii) The conditions in \eqref{YKcon} and \eqref{WTcon}  are quite similar; the lowest/highest initial slopes play the major role in the wave-breaking conditions. In addition to this, in \eqref{YKcon} one can observe \emph{some balance} between these slopes and the values of the fluid deflection $u$ evaluated at $\xi_i$. Indeed, the negative function $E$ in the right hand side of \eqref{YKcon} is decreasing in $u^M$ and increasing in $u^m$. This may be interpreted as, even if the lowest initial slope is sufficiently low, the wave breaking may not occur if the lowest initial slope is attained at too high or too low wave height. In contrast to this, the wave breaking condition for Whitham equation in \cite{CE98} showes that the wave breaking can occur at \emph{any height} as long as the initial slope satisfies the conditions in \eqref{WTcon}.\\
iii) In the proof of Theorem \ref{thm_wt}, we want to emphasize the method and not the technicalities. For this reason we considered the simplest kernel handled in \cite{CE98}. But our argument in the proof works equally well for the completed kernels handled in \cite{MLQ16} and generic smooth function $f(u)$.   Also,  the constants in the blowup condition could be optimized by more delicate estimate.\\

\section{Proof of Theorem \ref{thm_wt}}


In this section, we turn now to the proof of Theorem \ref{thm_wt}. We consider
\begin{equation}\label{WTP}
\left\{
  \begin{array}{ll}
     u_t + f(u)u_x + \int_{\mathbb{R}} K(x-y)u_y (t, y) \, dy =0, & t>0, x \in \mathbb{R}, \\
    u(0,x)=u_0 (x), &   x\in \mathbb{R}.\hbox{}
  \end{array}
\right.
\end{equation}
with $$f(u):=3\sqrt{1+u} -2.$$

Let $d=u_x$ and applying $\partial_x$ to the first equation of \eqref{WTP}, we obtain
\begin{equation}\label{3deqn}
\begin{split}
\frac{D}{Dt}d&:=(\partial_t + f(u)\partial_x)d\\
&=-f'(u)d^2 -\int_{\mathbb{R}} K(z)u_{xx}(t, x-z) \, dz.
\end{split}
\end{equation}
We then trace the evolution of $-f'(u)$ along the same characteristic. That is, we estimate
$$\frac{D}{Dt}(- f'(u))= (\partial_t + f(u)\partial_x )(-f'(u)).$$
\begin{lemma}\label{lemma3-1}
For $t>0$, it holds,
$$-\frac{3}{4} \frac{M}{(1+u)^{3/2}} \leq \frac{D}{Dt} (-f'(u)) \leq \frac{3}{4} \frac{M}{(1+u)^{3/2}},$$
where $M:=\|K' \|_{L^\infty} \|u_0 \|_{L^1}$.
\end{lemma}
\begin{proof}
It follows from the first equation in \eqref{WTP} that
$$\partial_t f'(u) = f''(u)u_t= f''(u)\{ -f(u)u_x  - \int_{\mathbb{R}} K(x-y) u_y (t,y) \, dy \}.$$
Since $f(u)\partial_x f'(u) = f(u)f''(u)u_x$, $\frac{D}{Dt} f'(u)=\partial_t f'(u) + f(u)\partial_x f'(u) $ is reduced to
$$\frac{D}{Dt} f'(u) = -f''(u) \int_{\mathbb{R}} K(x-y) u_y (t,y) \, dy.$$
Integration by parts and $f''(u)=-\frac{3}{4} \frac{1}{(1+u)^{3/2}}$ give the desired result.
\end{proof}
Unlike with the leading coefficient functions of several models in Section 2,  the rate of change of $-f'(u)$ along the characteristic does \emph{not} have uniform constant bounds, and depends on $u$. Tracing the evolution of $u$ using the first equation of \eqref{WTP}, along the same characteristic gives
\begin{equation}\label{3ub}
-M \leq \frac{D}{Dt} u \leq M,
\end{equation}
where $M$ defined above. Integrating \eqref{3ub} over $[0,t]$ gives
$$u\geq -M t + u(0).$$
Recalling the bounds in Lemma \ref{lemma3-1}, we get
\begin{equation}\label{3fpbound}
\frac{D}{Dt} (-f'(u)) \leq \frac{3}{4}\frac{M}{(-Mt + u(0) +1 )^{3/2}}.
\end{equation}
Note that the expression in the right hand side, is valid only when $-Mt + u(0) +1 >0$.

We now let
\begin{equation*}
\begin{split}
&m_1(t):=\inf_{x \in \mathbb{R}} [u_x(t,x)]=d(t, \xi_1(t)) \ and\\
&m_2(t):=\sup_{x \in \mathbb{R}} [u_x(t,x)]=d(t, \xi_2 (t)).\\
\end{split}
\end{equation*}
Evaluating the inequality in \eqref{3fpbound} at $x=\xi_i (t)$, $i=1,2$, we obtain
$$\frac{D}{Dt}\bigg{(} -f'\big{(}u(t, \xi_i (t)) \big{)} \bigg{)}\leq  \frac{3}{4} \frac{M}{(-Mt + u(0, \xi_i(0)) +1 )^{3/2}},   $$
or
\begin{equation}\label{3fpbound2}
\frac{D}{Dt} \bigg{(} -f'\big{(}u(t, \xi_i (t)) \big{)} \bigg{)} \leq \frac{3}{4}\frac{M}{(-M t + u^m _0 +1)^{3/2}}, \ i =1,2,
\end{equation}
where 
$$u^m _0 := \min_{i=1,2} \big{[} u(0, \xi_i (0)) \big{]}.$$
It follows that the inequality \eqref{3fpbound2} holds for $t \in [0, \frac{u^m _0 +1}{M})$.

We are now headed to find a short time \emph{negative} upper bound of $-f'(u(t, \xi_i (t)))$. We choose some 
\begin{equation}\label{3tildeM}
\accentset\sim{M} >\frac{3}{4} \frac{M}{(u^m _0 +1)^{3/2}},
\end{equation}
then \eqref{3fpbound2} gives
\begin{equation}\label{3fpbound3}
\frac{D}{Dt} \bigg{(} -f'\big{(}u(t, \xi_i (t)) \big{)} \bigg{)} \leq \frac{3}{4}\frac{M}{(-M t + u^m _0 +1)^{3/2}} < \accentset{\sim}{M},
\end{equation}
for $0\leq t< t^{**} $. Here 
\begin{equation}\label{3tss}
t^{**} = \frac{1}{M} \bigg{\{} 1+ u^m _0 -\bigg{(}\frac{3}{4} \frac{M}{\accentset\sim{M}} \bigg{)}^{2/3}   \bigg{\}}< \frac{u^m _0 +1}{M}.
\end{equation}
Integrating \eqref{3fpbound3} over $[0,t]$, we get
$$-f'(u(t, \xi_i (t))) \leq \accentset{\sim}{M} t - f'(u(0, \xi_i (0))), \ \ 0\leq t \leq t^{**}.$$
Since $-f'(u)=-\frac{3}{2} \frac{1}{\sqrt{1+u}}$, it follows that
\begin{equation}\label{3fpbound4}
-f'(u(t, \xi_i (t))) \leq \accentset{\sim}{M} t - \frac{3}{2} \frac{1}{\sqrt{1+u^M _0}}, \ \ 0\leq t \leq t^{**}, \ \ i=1,2,
\end{equation}
where 
$$u^M _0 := \max_{i=1,2} \big{[} u(0, \xi_i (0)) \big{]}.$$
Now we arrive at the key inequality; from \eqref{3fpbound4} we see that for 
$$ \mu \in \bigg{[}- \frac{3}{2} \frac{1}{\sqrt{1+u^M _0}} , 0 \bigg{)},$$
and 
\begin{equation}\label{3ts}
t^*= \frac{1}{\accentset{\sim}M}\bigg{(} \mu + \frac{3}{2} \frac{1}{\sqrt{1 + u^M _0}}  \bigg{)},
\end{equation}
$\mu$ serves as a short time negative upper bound of $-f'(u(t, \xi_i (t)))$, therefore
\begin{equation}
-f'(u(t, \xi_i (t))) \leq \mu <0, \ \ for \ \ t\in [0, t^*].
\end{equation}

Evaluating \eqref{3deqn} at $x=\xi_i (t)$, $i=1,2$, and using the above upper bound, we obtain
\begin{equation}\label{3meqn}
\frac{D}{Dt}m_i = \mu m^2 _i  -\int_{\mathbb{R}} K(z)u_{xx}(t, \xi_i (t)-z) \, dz, \ \ a.e. \ \ on \ \ (0, t^*), \ i=1,2.
\end{equation}

We are aim to find an initial profile of \eqref{WTP} that yields wave breaking. In order to do this, we seek the initial profiles that lead to the blow-up of $d=u_x$ \emph{before} $t^*$. 
That is, we shall need to i) find the blow-up condition of \eqref{3meqn}, ii) check if the blow-up occur before $t^*$. Furthermore, 
one need to verify $$t^* < t^{**},$$
with an appropriate choice of $\accentset{\sim}{M}$. From \eqref{3ts} and \eqref{3tss}, we observe that the inequality easily hold for sufficiently large $\accentset{\sim}{M}$. Later we will choose an \emph{optimally} \emph{small} $\accentset{\sim}{M}$, because larger $\accentset{\sim}{M}$ leads to smaller $t^*$, which in turn requires more extreme initial profiles for the blow-up.

Following the argument in \cite{CE98}, and from \eqref{3meqn} we infer the inequalities
\begin{equation}\label{3meqn2}
\begin{split}
&\frac{D}{Dt}m_1 \leq \mu m^2 _1 + K(0)(m_2 - m_1),  \ \ a.e. \ \ on \ \ (0, t^*),\\
&\frac{D}{Dt}m_2 \leq \mu m^2 _2 + K(0)(m_2 - m_1),  \ \ a.e. \ \ on \ \ (0, t^*).\\
\end{split}
\end{equation}
We are still following the argument in \cite{CE98}:
Summing up the two equations in \eqref{3meqn2}, we have
\begin{equation}
\begin{split}
\frac{D}{Dt} (m_1 + m_2) &\leq \mu(m^2 _1 + m^2 _2) + 2K(0)(m_2 - m_1)\\
&=(m_2 -m_1) \big{\{}2K(0) -\mu (m_1 + m_2)  \big{\}} +2\mu m^2 _2.
\end{split}
\end{equation}
If we assume that 
$$m_1(0) + m_2 (0) \leq \frac{2K(0)}{\mu},$$
we see that $m_1 (t) + m_2 (t)$ remains so for all time. Using this fact, we consider
\begin{equation}
\begin{split}
\frac{D}{Dt} m_1 &\leq \mu m^2 _1 + K(0)(m_2 - m_1)\\
&\leq \mu m^2 _1 + K(0) \bigg{(} \frac{2K(0)}{\mu} -m_1 \bigg{)} - K(0)m_1\\
&=\mu \bigg{(} m_1 - \frac{K(0)}{\mu}  \bigg{)}^2 + \frac{K^2 (0)}{\mu}.
\end{split}
\end{equation}
Defining $m(t):=m_1 (t) - K(0)/\mu$, and since $\mu <0$, we obtain
$$\frac{D}{Dt} m \leq \mu m^2, \ \ a.e. \ \ on \ \ (0, t^*).$$
Integration yields
\begin{equation}\label{3mineq}
m(t) \leq \frac{m(0)}{1-m(0)\mu t}.
\end{equation}
Therefore,
\begin{equation}
m(t)\rightarrow -\infty, \ \ as \ \ t\rightarrow \frac{1}{m(0)\mu}.
\end{equation}

The above argument can be summarized as follows:
\begin{lemma}\label{3lemma_m_b}
If $m_1(0)  \leq \frac{2K(0)}{\mu} -m_2 (0)$, then
$$m(t):=m_1 (t) -\frac{K(0)}{\mu} \rightarrow -\infty,$$
as $t \rightarrow \frac{1}{m(0)\mu}$.
\end{lemma}
\begin{proof}
Since $m_1(0)  \leq \frac{2K(0)}{\mu} -m_2 (0)$, we have $m_1(0) \leq \frac{2K(0)}{\mu}<0,$
because $m_2(0) \geq 0 $ and $\mu<0$. Furthermore, since
$m(0)=m_1(0) - \frac{K(0)}{\mu} \leq \frac{K(0)}{\mu}<0 $, from \eqref{3mineq}, we have 
$$m(t) \rightarrow -\infty$$
as $$t\rightarrow \frac{1}{m(0)\mu}.$$
\end{proof}

As noted earlier, in order to validate our argument, we need to show that the blow-up occur before $t^*$. In addition to this, it is required to prove $t^* \leq t^{**}$. These will be fulfilled by choosing an appropriate $\accentset{\sim}{M}$ in \eqref{3tildeM}. We first state and prove some elementary inequality:
\begin{lemma}\label{3lem3.2}
For any $A, B >0$, if 
$$\kappa:=\frac{\frac{3}{4}\frac{1}{\sqrt{A}}}{B+\frac{3}{4} \frac{1}{\sqrt{A}}},$$
then 
$$\frac{3}{4} \frac{1}{(\kappa A)^{3/2}} > \frac{B}{A-A\kappa}.$$
\end{lemma}
\begin{proof}
Simplifying the following fraction
\begin{equation}\label{3lemexp}
\frac{3}{4} \frac{1}{(\kappa A)^{3/2}} \bigg{/} \frac{B}{A-A\kappa},
\end{equation}
gives
$$\frac{3}{4} \frac{1}{\sqrt{A}} (1-\kappa) \bigg{/} B\kappa^{3/2}.$$
Substituting the given $\kappa$ into above, we get
$$\frac{\sqrt{B + \frac{3}{4}\frac{1}{\sqrt{A}} }}{\sqrt{\frac{3}{4} \frac{1}{\sqrt{A}}}},$$
which is strictly greater than $1$. This completes the proof.
\end{proof}

We now let
$$A:=u^m _0 +1, \ B:=\mu +\frac{3}{2}\frac{1}{\sqrt{1+u^M _0}},$$
and
$$\accentset{\sim}{M}:=\frac{3}{4} \frac{M}{\big{\{}\kappa(u^m _0 +1) \big{\}}^{3/2}}, $$
with $\kappa$ is defined in the above lemma. i.e.,
$$\kappa=\frac{\frac{3}{4}\frac{1}{\sqrt{A}}}{B+\frac{3}{4} \frac{1}{\sqrt{A}}}=\bigg( \frac{4}{3} \mu \sqrt{1+u^m _0}  +2\frac{\sqrt{1+u^m _0}}{\sqrt{1+u^M _0}} +1  \bigg{)}^{-1}.$$
Then since,
$$t^{**} = \frac{1}{M} \bigg{\{} 1+ u^m _0 -\bigg{(}\frac{3}{4} \frac{M}{\accentset\sim{M}} \bigg{)}^{2/3}   \bigg{\}}, \ t^*= \frac{1}{\accentset{\sim}M}\bigg{(} \mu + \frac{3}{2} \frac{1}{\sqrt{1 + u^M _0}}  \bigg{)}$$
$$$$
from \eqref{3tss} and \eqref{3ts}, we have
$$\frac{t^{**}}{t^*}=\frac{\accentset\sim{M}}{M} \bigg{/} \frac{B}{B-A\kappa},$$
which is equivalent to \eqref{3lemexp}. Therefore, applying Lemma \ref{3lem3.2}, we obtain
$$t^{**} > t^{*}.$$

Next we show that the blow-up occur before $t^*$. More precisely, we show that if
\begin{equation}\label{3mcondition}
m_1(0) \leq \min\bigg{[} \frac{2K(0)}{\mu} - m_2(0), \frac{1}{\mu} \bigg{(}\frac{3}{4} \frac{M}{(\kappa(u^m _0 +1))^{3/2}} \bigg{)} \bigg{(}\mu + \frac{3}{2}\frac{1}{\sqrt{1+u^M _0}}\bigg{)}^{-1} +\frac{K(0)}{\mu} \bigg{]},
\end{equation}
then $m(t) \rightarrow -\infty$ as $t\rightarrow t^*$. From Lemma \ref{3lemma_m_b}, we observe that $m_1(0) \leq \frac{2K(0)}{\mu} -m_2 (0)$ leads to the blow-up of $m(t)$ as $t \rightarrow \frac{1}{m(0)\mu}$. Therefore, it suffices to show that $\frac{1}{m(0)\mu} \leq t^*$.

From \eqref{3mcondition}, since $m(0) = m_1 (0) - \frac{K(0)}{\mu}$ we have
$$m(0) \leq \frac{1}{\mu} \bigg{(}\frac{3}{4} \frac{M}{(\kappa(u^m _0 +1))^{3/2}} \bigg{)} \bigg{(}\mu + \frac{3}{2}\frac{1}{\sqrt{1+u^M _0}}\bigg{)}^{-1}.$$
By rearranging, 
\begin{equation*}
\begin{split}
\frac{1}{m(0)\mu} &\leq \frac{4(\kappa(u^m _0 +1))^{3/2}}{3M} \bigg{(}\mu + \frac{3}{2}\frac{1}{\sqrt{1+u^M _0}}\bigg{)}\\
&=\frac{1}{\accentset{\sim}{M}}\bigg{(}\mu + \frac{3}{2}\frac{1}{\sqrt{1+u^M _0}}\bigg{)}=t^*.
\end{split}
\end{equation*}
Hence, we complete the proof of Theorem \ref{thm_wt}.

\bibliographystyle{abbrv}

\end{document}